\documentclass[12pt]{article}
\usepackage[latin1]{inputenc} % accents
\usepackage{latexsym}
\usepackage{csvsimple}
\usepackage{bbm}              % fontes doubles (pour les ensembles, par ex.)
\usepackage{dsfont}           % fontes doubles (pour les ensembles, par ex.)
\usepackage{graphicx}         % pour d'eventuelles figures
\usepackage{epsfig}           % (preferer graphicx, si possible)
\usepackage{amsmath}
\usepackage{amsthm}         % AMSTEX
\usepackage{amsfonts}
\usepackage{amssymb}
\usepackage{empheq}
\usepackage{color}
\usepackage{array}
\usepackage{amsmath}
\usepackage{framed}
\usepackage{fancyhdr}
\usepackage{array,multirow,makecell}
\usepackage{multicol}
\usepackage{titlesec}
\usepackage{tikz}
\usepackage{times}
\usepackage{graphics}
\usepackage[T1]{fontenc}
\usepackage{algorithm}
 \usepackage{algorithmic}
 \usepackage{caption}
\usepackage[colorlinks,citecolor=red,urlcolor=blue,bookmarks=false,hypertexnames=true]{hyperref}
\newcolumntype{L}[1]{>{\raggedleft\arraybackslash   }b{#1}}

\theoremstyle{plain}
\newtheorem{theorem}{Theorem}[section]

\newtheorem{prop}[theorem]{Proposition}
\newtheorem{prof}[theorem]{Proof}
\newtheorem{remark}[theorem]{Remark}

\theoremstyle{definition}
\newtheorem{definition}[theorem]{Definition}

\begin{document}
%\maketitle

%\vspace{-20pt}

%
%\hrule

\begin{center} {\Large\bf Solving Nonsmooth Bi-Objective Environmental and
Economic Dispatch Problem using Smoothing
Techniques}
\end{center}
\begin{center}
{\footnotesize Mohamed Tifroute$^1$,  Anouar Lahmdani$^2$ and Hassane Bouzahir$^3$ \\
E-mails: mohamed.tifroute@gmail.com, anouar.lahmdani@gmail.com
}\end{center}

\vspace{30pt}

\begin{center} {\footnote{Higher School of Technology - Guelmim , Ibn Zohr University, Morocco}, \footnote{
FSA - Ait Melloul, Ibn Zohr University, Morocco} and \footnote{ENSA, Ibn Zohr University, Morocco } }
\end{center}

\clearpage

\begin{abstract}
The Environmental and Economic Dispatch problem (EEDP)\cite{A42, A99} is a nonlinear Multi-objective Optimization Problem (MOP) which simultaneously satisfies multiple contradictory criteria, and it's a nonsmooth problem when valvepoint effects, multi-fuel effects and prohibited operating zones have been considered. It
is an important optimization task in fossil fuel fired power plant operation for
allocating generation among the committed units such that fuel cost and pollution (emission level) are optimized simultaneously while satisfying all operational
constraints. In this paper, we use smoothing functions with the gradient
consistency property to approximate the nonsmooth   multi-objective
Optimization problem. Our approach is based on the smoothing method.
In fact, we explain the convergence analysis of smoothing method by
using  approximate Karush-Kuhn-Tucker condition, which is
necessary for a point to be a local weak efficient solution and is
also sufficient under convexity assumptions. Finally, we give an
application of our approach for solving the bi-objective EEDP.
\end{abstract}
\vspace{30pt}


%\medskip
%\noindent
%\subjclass{\footnotesize {\bf Mathematical subject classification:}
%Primary: 65K05; Secondary: 90C30.}

%\medskip
%\noindent

%\medskip
%
%\hrule
\begin{sloppypar}
\section{Introduction}
During the last decades  the area of nonsmooth (nondifferentiable)
Multiobjective Optimization Problems (MOP) has been extensively
developed. The MOP  refers to the process of simultaneously
optimizing two or more real-valued objective functions. For
nontrivial problems, no single point will minimize all given
objective functions at once, and so the concept of optimality is to
be replaced by the concept of Pareto optimality or efficiency. One
should recall that a point is called Pareto optimal or efficient, if
there is no different point with the same, or smaller, objective
function values, such that there is a decrease in at least one
objective function value. The nonsmooth MOP problem has applications
in engineering \cite{x1}, economics \cite{x2},
mechanics \cite{x3} and other fields. For more
details,see, for example, Miettinen
\cite{11}.\\

 In this paper, we concentrate on solving a classe of nonsmooth  MOP  that include $\min$, $\max$,  absolute value functions or composition of the plus function
 with smooth functions. Which the approximations are constructed based on the smoothing function for the plus function. For this end we introduce the concept of
 approximate Karush-Kuhn-Tucker AKKT condition for
the approximate  multiobjective problem  inspired by Giorrg. G et
al. \cite{4} and we adapt it to prove the convergence analysis of the
smoothing method, whose feasible set is  defined by inequality
constraints. Note that the AKKT condition has been widely used to
define the stopping criteria of many practical contrained
optimization algorithms \cite{v2000,v2006,v2014}. The objective is to update the smoothing parameters to guarantee the
convergence. We point out that Chen \cite{3} has dealt with
convergence analysis of smoothing method (in the scalar case) by
using a gradient method. Finally, we give an application of
our approach in solving bi-objective Economic and Environmental
Dispatching Problem (EEDP)\cite{A42}. In fact, we transform the nonsmooth EEDP into a set of single-objective subproblems using  the $\epsilon$-constraint method. The objective function of the subproblems is smoothed  and the subproblems are solved by the interior point barrier method.

This paper is organized as follows. In Section 2, we state the
problem under consideration and we recall some useful basic
notations. In Section 3, we define a class of smoothing composite
functions by using the plus function. In Section 4, to explain the
convergence analysis of the smoothing method, we use sequential
AKKT. Finally, we show a numerical application
in Section 5.
\section {Basic notations and properties}
The following notations are used throughout this paper. By
$\langle\cdot,\cdot\rangle$, we denote the usual inner product on
$\mathbb{R}^n$, and by $\|\cdot\|$ we denote its corresponding norm.
Let $\mathbb{R}_{+}^p=\{ x\in \mathbb{R}^p ~:~~ x_i\geq 0, ~~
i={1,\cdots,p}\},~\mathbb{R}_{-}^p=\{ x\in \mathbb{R}^p~:~~ x_i\leq
0, ~~ i={1,\cdots,p}\},~\mathbb{R}_{++}^p=\{ x\in \mathbb{R}^p ~:~~
x_i>0, ~~ i={1,\cdots,p}\}$ and $\mathbb{R}_{--}^p=\{ x\in
\mathbb{R}^p ~:~~ x_i<0, ~~ i={1,\cdots,p}\}$, we consider the
partial orders $\succeq$ (respectively, $\preceq$) and $\succ$ (respectively,
$\prec$), defined as $x \succeq y$ (respectively, $ x \preceq y$) if and only
if $x-y\in \mathbb{R}_+^p $ (respectively, $x-y\in \mathbb{R}_-^p $) and $x
\succ y$ (respectively, $x \prec y$) if and only if $x-y\in \mathbb{R}_{++}^p
$ (respectively, $x-y\in \mathbb{R}_{--}^p $). In this paper we consider the
 nonsmoothing multiobjective problem
  \[
 P_{(1)}: ~\left\{
   \begin{array}{ll}
   \displaystyle \min F (x),\\
    \text{subjet to~~}  x\in S.
   \end{array}
 \right.
 \]\

where $ S=\{ x\in \mathbb{R}^n ~~/~~ g_j(x)\leq 0,~~
j=1,\cdots,m \}$, the objective function $F
:~\mathbb{R}^n\rightarrow \mathbb{R}^p $ is given by
$F(x)=(f_1(x),\cdots,f_p(x))$ nonsmooth, convexe and locally lipschitz and $g_j :~\mathbb{R}^n\rightarrow
\mathbb{R}  is continuously differentiable, ~~j=1,\cdots,m $, $S$ is a feasible set of
$P_{(1)}$. The set of active indexes at a point $x\in S$ is given by
$J(x)=\{ j, ~~ g_j(x)=0\}$  . A point $x^*\in S$ is called Pareto
optimal point or (efficient solution)  of problem $P_{(1)}$ if there
exists no other $x \in S$ with $f_i(x)\leq f_i(x^*),
~~i=1,\cdots,m$ and $f_j(x)<f_j(x^*)$ for at least one index
$j$. If there exists no $x\in S$ with
$f_i(x)<f_i(x^*)~~i=1,\cdots,p$, then $x^*$ is said to be a weak
Pareto optimal point or (weak efficient solution) of problem
$P_{(1)}$.

\begin{definition}\cite{B1}
The upper Clarke directional derivative of a locally Lipschitz
function $f :~\mathbb{R}^n\rightarrow \mathbb{R} $ at $x$ in the
direction $d\in \mathbb{R}^n$ is $$ f^{\circ}(x,d)=\limsup_{z\rightarrow x,~t\downarrow 0} \frac{f(z+td)-f(y)}{t}$$ and
the Clarke subdifferential of $f$ at $x$ is given by
$$\partial_{c}f(x)=\{ \lambda \in \mathbb{R}^n :~~\langle \lambda, d\rangle \leq
f^{\circ}(x,d)~~ \forall d\in \mathbb{R}^n\}$$
\end{definition}
When $f$ is continuously differentiable, one has
$\partial_{c}f(x)=\{\nabla f(x)\}$. Now, we recall some
results which will be needed in our
convergence analysis.
\begin{prop}\cite{B1}\label{p1}

      Let $f :~\mathbb{R}^n\rightarrow \mathbb{R} $ be locally Lipschitz and $h :~\mathbb{R}^n\rightarrow \mathbb{R}$
continuously differentiable. Then
\begin{description}
\item[(i)] $\partial_{c}(f(x)+h(x))=\partial_{c}f(x)+ \nabla h(x)$.
   \item[(ii)] If $x^*$ is a local minimum of $f$, then $0\in \partial_{c}f(x)$.
    \item[(iii)]  If $f(x) = \max\{f_1(x), \cdots , f_p(x)\}$, where $f_j :~\mathbb{R}^n\rightarrow \mathbb{R} $
for all $j\in \{1,\cdots,p\}$ are continuously differentiable, then
$$\partial_{c}f(x) = conv\{\nabla f_j (x) ~:~~ j = 1, \cdots , p
\text{\quad such  that~}  f_j (x) = f(x)\}$$ (here \textbf{conv} denotes
the convex hull).
\end{description}

\end{prop}

\section{Smoothing function}
Rockafellar and Wets have shown that for any locally Lipschitz function
$f$, we can construct a smoothing function by using the convolution $$
f(x,\mu)=
\int_{\mathbb{R}^n}f(x-y)\psi_{\mu}(y)dy=\int_{\mathbb{R}^n}f(y)\psi_{\mu}(x-y)dy$$
where $ \psi_{\mu}:~\mathbb{R}^n \rightarrow \mathbb{R}$ is a smooth
kernel function, (see \cite{B3}). In this section we extend the
smoothing method given by Chen \cite{3} to solve nonsmooth MOP, for this, we start by considering a class of smoothing functions.
\begin{definition}
Let $F :~\mathbb{R}^n\rightarrow \mathbb{R}^p $ be a continuous
function given by $F(x)=(f_1(x),\cdots,f_p(x))$, we define a smoothing
function of $F$  by  $\widetilde{F} :~\mathbb{R}^n \times
\mathbb{R}_{++}^p\rightarrow \mathbb{R}^p $ where
$\widetilde{F}(x,\mu)=(\widetilde{f}_1(x,\mu_1),\cdots,\widetilde{f}_p(x,\mu_p))$
such that for each $i=1,\cdots,p~~ \widetilde{f}_i(x,\mu_i) $ is
continuously differentiable in $\mathbb{R}^n$ for any fixed $\mu_i
\in \mathbb{R}_{++}^p$, and for any $x\in \mathbb{R}^n$
$$ \lim\limits_{y \longrightarrow x,\mu_i\downarrow 0}\widetilde{f}_i(y,\mu_i)=f_i(x)~~~ $$
\end{definition}

Now we can construct a smoothing method by using $\widetilde{F}$ and
$\nabla\widetilde{F}$ as follows. The first step is to define a
parametric smooth function $\widetilde{F}(x ,\mu_k)$ to approximate
$F(x)$. The second step we find for a fixed $\mu_k \in \mathbb{R}_{++}^p$ an approximate solution of the smooth MOP
 \[
 P_{({\mu_k})}: ~\left\{
   \begin{array}{ll}
   \displaystyle \min \widetilde{F} (x,\mu_k),\\
    \text{subjet to}~  x\in S.
   \end{array}
 \right.
 \]\

The last step, by updating $\mu_k $, which guarantees the convergence of any accumulation point of a designated subsequence of the iteration sequence generated by the smoothing MOP algorithm is a AKKT point. So
the Pareto optimal solutions (stationary points) of the approximate subproblems $P_{({\mu_k})}$ converge to a Pareto optimal solution (stationary point) of the initial MOP $P_{(1)}$. Note that
the advantage of the smoothing method is to solve optimization problems with
continuously differentiable functions which has
a rich theory and powerful methods \cite{B2}.\\

Many nonsmooth optimization problems can be reformulated by using the
plus function $(h)_+$ for exemple
$\max(h,g)=h+(g-h)_+,~~\min(h,g)=h-(h-g)_+$ and $|h|=(h)_++(-h)_+$.
So that, in this paper, we present a class of smooth
approximation for the plus function by convolution given by Chen
\cite{3}.

%and Mangisaria 17 90 94 2012

\begin{definition}\cite{3}\label{d1}
Let $\rho :~ \mathbb{R}\rightarrow \mathbb{R}^+$ be a piecewise
continuous density function satisfying $$
\rho(s)=\rho(-s)~~~~and~~~~~\kappa :=\int_{\mathbb{R}}|s|\rho(s) ds<
\infty$$
 then the function $\phi :~ \mathbb{R}\times \mathbb{R}^+\rightarrow
 \mathbb{R}^+$ defined by
 \begin{equation}\label{S1}
\phi(h,\mu):=\int_{\mathbb{R}}(t-\mu s)_+\rho(s) ds
\end{equation}
is a smoothing function of $(h)_+$.
\end{definition}

\begin{prop} \cite{3}\label{p1}

For any fixed $\mu >0$, $\phi(\cdot,\mu)$ is continuously
differentiable convex, strictly increasing, and satisfies
 \begin{equation}\label{S2}
0<\phi(h,\mu)-(h)_+\leq \kappa \mu
\end{equation}
then for any $h\in \mathbb{R}$
\begin{equation}\label{S3}
\lim\limits_{h_k\rightarrow h~~\mu\downarrow 0} \phi(h_k,\mu)=(h)_+
\end{equation}
\end{prop}

\begin{prop} \cite{3}\label{p2}
Let $\partial(h)_+$  the Clarke subdifferential of $(h)_+$ and
$G_{\phi}(h)$ is the  subdifferential associated with the smoothing
function $\phi$ at $h$ given by
\begin{equation}\label{S4}
G_{\phi}(h)=con\{ \tau ~/~ \nabla_t\phi(h_k,\mu_k)\rightarrow
\tau~,~~h_k\rightarrow h,~~ \mu_k\downarrow 0\}
\end{equation}
then $$G_{\phi}(h)=\partial(h)_+$$
\end{prop}

\begin{remark}

The plus function $(h)_+$ is convex and globally Lipschitz
continuous. Any smoothing function $\phi(h,\mu)$ of $(h)_+$  is also
convex and globally Lipschitz. In addition, for any fixed $h$, the function
$\phi$ is continuously differentiable, monotonically increasing and
convex with respect $\mu>0$
 and satisfies
$$0\leq
\phi(t,\mu_2)-\phi(t,\mu_1)\leq\kappa(\mu_2-\mu_1)~~~~~\text{for}~~~~\mu_2>\mu_1$$

\end{remark}

Now, we study properties of the smoothing function $\phi$. We assume
that $F:~\mathbb{R}^n\rightarrow \mathbb{R}^p$ given by
$F(x)=(f_1(x),\cdots,f_p(x))$ is locally Lipschitz continuous.
According to Rademacher's theorem, $F$ is differentiable almost
everywhere. For each $i=1,\cdots,p $the Clarke subdiferential
of $f_i$ at a point $x$ is defined by $$\partial f_i(x)=conv\{ v~/~
\nabla f_i(z)\rightarrow v, ~~ f_i ~\text{is differentiable at}
~z,~~z\rightarrow x\}$$

For a locally Lipschitz function $f_i$, the gradient consistency
$$\partial f_i(x)=conv\{ \lim\limits_{x_k\rightarrow
x~~\mu^{i}_k\downarrow 0}\nabla
{\widetilde{f}_i}(x_k,\mu^{i}_k)\}=G_{\widetilde{f_i}}(x)~~~~\forall
x\in \mathbb{R}^n$$ between the Clarke subdifferential and
subdiferential associated with the smoothing function of $f_i$ for each
$i= 1,\cdots,p$. Note that the abrove result is important for the convergence
of smoothing methods.\\

Throughout the rest of this paper we assume that the function $F$ is
given by $F(x)=H((\varphi(x))_+)$ where $H(x)$ and $\varphi(x)$ are
continuously differentiable,  $H(x)=(h_1(x)\cdots,h_p(x))$ with
composents
 $h_i:~\mathbb{R}^n\rightarrow \mathbb{R},~~i= \{1,\cdots,p\}$ and
$\varphi(x)=(\varphi_1(x)\cdots,\varphi_n(x))$ with
$\varphi_j:~\mathbb{R}^n\rightarrow \mathbb{R},~~j=
 1,\cdots,n$. Notice that  $
 (\varphi(x))_+=((\varphi_1(x))_+,\cdots,(\varphi_n(x))_+)$ and its smoothing
 function is  $
 \phi(\varphi(x),\mu)=(\phi(\varphi_1(x),\mu),\cdots,\phi(\varphi_n(x),\mu))^T$.

 Now we show the gradient consistency of the smoothing
composite functions using $\phi$ in definition \ref{d1}  for the plus
function.

 \begin{theorem}\label{t1}
% For each$i\in \{1,\cdots,p\}$, we have.\\
 Let $F(x)=H((\varphi(x))_+)$, where $\varphi:~\mathbb{R}^n\rightarrow
 \mathbb{R}^n$ and $H:~\mathbb{R}^n\rightarrow \mathbb{R}^p$ are
 continuously differentiable, then for each $i= 1,\cdots,p$,
 $\widetilde{f}_i(x,\mu^{i}_k)=h_i(\phi(\varphi(x)),\mu^{i}_k) $ is a smoothing
 function of $f_i$ with the following properties.
 \begin{description}
    \item \textbf{(i)} For any $x\in \mathbb{R}^n$ , $\{\lim\limits_{x_k\rightarrow
x~~\mu^{i}_k\downarrow 0}\nabla {\widetilde{f}_i}(x_k,\mu^i_k)\}$ is
nonempty and bounded, and $\partial f_i(x)=G_{\widetilde{f_i}}(x),$
for each $i= \{1,\cdots,p\}$.
    \item \textbf{(ii)} If $H$, $\varphi_j$  are convex for each $j\in \{1,\cdots,n\}$ and $\varphi_j$ is
    monotonically nondecreasing, then for any fixed $\mu_k^i\in \mathbb{R}_{++}^p $,
    $\widetilde{f}_i(\cdot,\mu_k^i)$ is convex.
 \end{description}
\end{theorem}

\begin{prof}
For any fixed $i\in\{1,\cdots,p\}$, we can derive this theorem by
theorem 1 \cite{3}.
\end{prof}
\begin{prop}\label{p9}
Let $\vartheta(t)=|t|$,
$\vartheta_\mu(t)=sin(\mu).ln(\cosh (\frac{t}{\sin(\mu)})),$
\\ $0 <\mu <\frac{\pi}{2}$. Then
\begin{description}
 \item(i) $0\leq{\vartheta}(t)-\vartheta_\mu(t)\leq\sin(\mu)\ln(2).$
  \item(ii) $|\frac{d({\vartheta_\mu}(t))}{dt}|<1$, and   $\frac{d({\vartheta_\mu}(t))}{dt}|_{t=0}=0$.
  \item (iii) $\vartheta_\mu(t)$ is convex.
\end{description}
\end{prop}
\begin{proof}
\begin{description}
\item $(i)~$Let
\begin{align*}
\vartheta_\mu(t)-\vartheta(t)&=sin(\mu).ln(\cosh (\frac{t}{\sin(\mu)}))-|t|\\
&=\sin(\mu).ln(\frac{1}{2}\exp(\frac{t}{\sin\mu})+\frac{1}{2}\exp(\frac{-t}{\sin\mu}))-|t|\\
&=\sin(\mu).ln(\exp(\frac{t-|t|}{\sin\mu})+\exp(\frac{-t-|t|}{\sin\mu}))+\sin(\mu).ln(\frac{1}{2})
\end{align*}

 %\begin{tabular}{p{8cm}p{20cm}}
%$\vartheta_\mu(t)-\vartheta(t)$ & $=\sin(\mu)
%\ln(\frac{1}{2}\exp(\frac{t}{\sin(\mu)})+\frac{1}{2}\exp(\frac{-t}{\sin(\mu)}))-|t|$
%& $=\sin(\mu)
%\ln(\exp(\frac{t-|t|}{\sin(\mu)})+\exp(\frac{-t-|t|}{\sin(\mu)})+\sin(\mu)
%\ln(\frac{1}{2})$.
%\end{tabular}
Since
$1<\exp(\frac{t-|t|}{\sin\mu})+\exp(\frac{-t-|t|}{\sin\mu})\leq 2$. Thus,\\
\begin{center}
\scalebox{0.95}{$\sin(\mu).ln(\frac{1}{2})<\sin(\mu).ln(\exp(\frac{t-|t|}{\sin\mu})+\exp(\frac{-t-|t|}{\sin\mu}))+\sin(\mu).\ln(\frac{1}{2})\leq
\sin(\mu).ln(2)+\sin(\mu).ln(\frac{1}{2})$}
\end{center}
then
$$\sin(\mu).\ln(\frac{1}{2})<\vartheta_\mu(t)-\vartheta(t)\leq 0.$$
Hence, we obtain
$$0\leq{\vartheta}(t)-\vartheta_\mu(t)\leq\sin(\mu)\ln(2).$$

Therefore, $\vartheta_{\mu}$ is a smoothing
approximation function $\vartheta$.
\item $(ii)$ We have
$$\frac{d\vartheta_\mu(t)}{dt}=\frac{\exp(\frac{t}{\sin\mu})-\exp(\frac{-t}{\sin\mu})}{\exp(\frac{t}{\sin\mu})+\exp(\frac{-t}{\sin\mu})}
=\frac{\exp(\frac{2t}{\sin\mu})-1}{(\exp\frac{2t}{\sin\mu})+1}$$

Since~~
$$\left|\frac{\exp(\frac{2t}{\sin\mu})-1}{\exp(\frac{2t}{\sin\mu})+1}\right|<1
$$

Then $$|\frac{d({\vartheta_\mu}(t))}{dt}|<1
$$
\item $(iii)~$
\begin{align*}
{\vartheta_\mu}^{\prime\prime}(t)& = \frac{d}{dt}\left[\frac{d({\vartheta_\mu}(t))}{dt}\right] \\ &=\frac{\frac{4}{\sin\mu}\exp(\frac{2t}{\sin\mu}))}{(\exp(\frac{2t}{\sin\mu})+1)^2}>0. \quad \forall \mu  \in ]0,\frac{\pi}{2}[.
\end{align*}

 Thus ${\vartheta_\mu}$ is convex.

\end{description}
\end{proof}
\begin{prop}\label{p10}
Let  $\vartheta (t)=|t|$, and a vector function $g(x) =
(g_1(x),\cdots , g_p(x))^T$ with components $g_j : \mathbb{R}^n
\rightarrow \mathbb{R}$, we denote $\vartheta(g(x))=|g(x)| =
(|g_1(x)|,\cdots , |g_p(x)|)$ and $\vartheta_{\mu}(g(x)) =
(\vartheta_{\mu_1}(g_1(x)),\cdots , \vartheta_{\mu_p}(g_p(x)))^T$,
with
$\vartheta_\mu(g_j(x))=sin(\mu_j).ln(\frac{1}{2}\exp(\frac{g_j(x)}{sin\mu_j})+\frac{1}{2}\exp(\frac{-
g_j(x)}{sin\mu_j}))$ for $j=\{1,\cdots, p\}$ and $0 <\mu_j
<\frac{\pi}{2}$. Then $\vartheta_{\mu}(g(x))$ is a smoothing
approximation function of $\vartheta(g(x))$.
\end{prop}
\begin{proof}
For any fixed $j\in\{1,\cdots, p\}$, we can derive by proposition
\ref{p10} that

\begin{equation*}
0<\vartheta(g_j(x))-\vartheta_{\mu_j}(g_j(x))\leq\sin(\mu_j).ln(2).
\end{equation*}
Considering  $\kappa=ln(2)$ and $0\leq\sin(\mu_j)\leq\mu_j$  $\forall \mu_j \in [0, \frac{\pi}{2}]$.\\
then,\\
\begin{equation*}
0<\vartheta(g_j(x))-\vartheta_{\mu_j}(g_j(x))\leq \kappa \mu
~~\text{for}~~j\in\{1,\cdots, p\}.
\end{equation*}
Therefore, $\vartheta_{\mu}(g(x))$ is a smoothing
approximation function $\vartheta(g(x))$.

\end{proof}
\section{Smoothing Multiobjective Optimization Problem}
In this section,  we introduce AKKT condition for the multiobjective problem
$P_{(\mu)}$ inspired by Giorrg. G et al. \cite{4}. Then  we exploit
it to prove  convergence analysis of the smoothing method, whose
feasible set is  defined by inequality constraints. In fact, the
solution of problem $P_{(1)}$ is accomplished by solving a sequence
of problems $P_{(\mu)}$, where the value of $\mu$ is updated
according to $\mu_{k+1}=\alpha \mu_k$ with $\alpha\in (0,1)$ is
the decreasing factor of $\mu$. We point out that Chen \cite{3} is
concerned with convergence analysis of the smoothing method (in the
scalar case) by using a smoothing gradient method.
\begin{definition}
We say that the AKKT condition is
satisfied for problem $P_{(\mu)}$ at a feasible point $x^*\in S$ if
there exists a sequence $(x_k)\subset \mathbb{R}^n$ and
$(\lambda_k,\beta_k)\subset \mathbb{R}^p\times \mathbb{R}^m$ such
that
\begin{description}
    \item $(C_0)~~ x_k \rightarrow x^*$
    \item $(C_1)~~ \sum \limits ^{p}_{i=1} \lambda^{i}_k \nabla \widetilde{f_i}(x_k,\mu^{i}_k)+ \sum^{m}_{j=1} \beta^{j}_k \nabla g(x_k)\rightarrow 0$
    \item $(C_2) ~~ \sum \limits^{p}_{i=1} \lambda^{i}_k =1$
    \item $(C_3)~~  g_j(x_k)< 0 \Rightarrow \beta^{j}_k =0$  for $j= \{1,\cdots,n\}$.
\end{description}

\end{definition}

\begin{remark}
\begin{description}
    \item (i) A point satisfying the AKKT
    is called AKKT point.
    \item (ii)  The sequence of points $(x_k)$ is not required to be
    feasible.
    \item (iii)  Assuming $\beta_k\in \mathbb{R}_+^m$, condition $(C_3)$ is
    equivalent to
    $$ \beta^{j}_k g_j(x_k)\leq 0 \text{~for sufficiently large~}
    k, ~~\forall j \notin J(x^*).$$
Each of these condition implies the condition
    $$ \beta^{j}_k  g(x_k)\rightarrow 0~~~ \forall j \notin J(x^*)$$
\end{description}

\end{remark}
The following theorem  establish necessary optimality conditions for
problem $P_{(1)}$.
\begin{theorem}
If $x^*\in S$ is a locally weakly efficient solution for Problem
$P_{(1)}$, then $x^*$ satisfies the AKKT condition with sequences
$(x_k)$ and $(\lambda_k,\beta_k)$. In addition, for these sequence
we have that $\beta_k=b_k(g(x_k))_+$ where $b_k>0 ~~ \forall k.$
\end{theorem}

\begin{prof}
Its proof is based on Theorem 3.1 in \cite{4}, the gradient consistency Theorem
\ref{t1} and Proposition \ref{p1}.
\end{prof}
In order to establish the sufficient condition
 we  assume convexity assumption and the
following condition.\\
\textbf{Assumption A}:~We call a sum converging to zero, if $~~~~
\sum^{m}_{j=1} \beta^{j}_k g(x_k)\rightarrow 0$.

\begin{theorem}
Assume that $H$, $\varphi_j$ for $j= \{1,\cdots,n\}$ and $g_i$ for
$i=\{1,\cdots,m\}$ are convex and $\varphi_j$ is
    monotonically nondecreasing. If $x^*\in S$  satisfies the AKKT
    condition and  assumption A is fulfilled then $x^*$is a global
    weak efficient solution of problem
$P_{(1)}$.
\end{theorem}

\begin{prof}
By Theorem 3.2 in \cite{4} and gradient consistency Theorem
\ref{t1}, we derive this Theorem.
\end{prof}
\newpage
\section{Bi-Objective Nonsmooth Environmental and Economic Dispatch Problem}
\begin{figure}
  \centering
  \includegraphics[width=10cm]{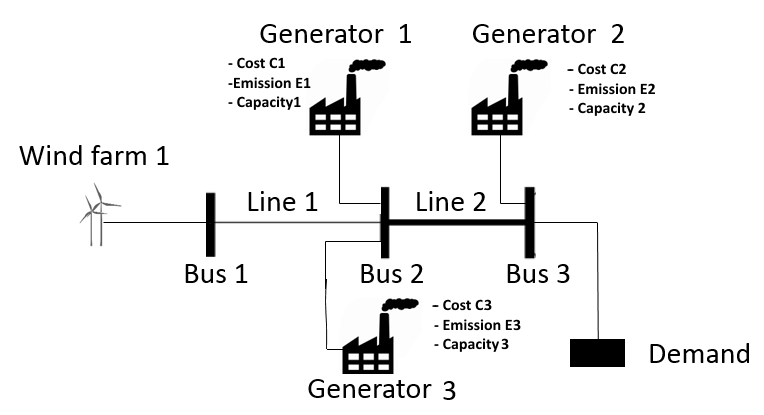}\\
  \caption{Bi-objective Non-smooth
Environmental and Economic Dispatch Problem.}\label{GGG}
\end{figure}
The Bi-objective Economic and Environmental Dispatch Problem (EEDP) is concerned with
the minimization of generation costs and the emission of pollutants
while representing systems operational constraints. Note that the
two objectives are conflicting in nature and they both have to be
considered simultaneously to find overall optimal dispatch. The
EEDP is a multi-objective, nonlinear, and nonsmooth problem.\\

\subsection{Notation}

 \begin{tabular}{p{2cm}p{16cm}}
      $C(P)$  & function cost for all thermal units [\$];\\
     $E(P)$ & total pollutant emission for all thermal units given in [kg/h];\\
    $P$   & vector of active power outputs for all thermal units
[MW]; \\
    $P_{i}$  & active power output of generating unit $i$ [MW]; \\
    $E_{i}$  & active emission output of generating unit $i$ [MW]; \\
     $P^{\max}_i$  & maximum power output of generating unit $i$ [MW]; \\
 $P^{\min}_i$ & minimum power output of generating unit $i$ [MW];\\
 $E^{\max}_i(P)$ & maximum pollutant emission given in [kg/h];\\
 $E^{\min}_i(P)$ & minimum  pollutant emission given in [kg/h].
    \end{tabular}\\

%\begin{figure}
%  \centering
%  % Requires \usepackage{graphicx}
%  \includegraphics[width=12cm]{GGG.png}\\
%  \caption{Bi-Objective Non-smooth
%Environmental and Economic Dispatch Problem}\label{GGG}
%\end{figure}

\subsection{Cost function}
The growing costs of fuels and operations of power generating units
require a development of optimization methods for Economic Dispatch
(ED) problems. Standard optimization techniques such as direct
search and gradient methods often fail to find global optimum
solutions.  The realistic operation of the ED problem considers
the couple  valve-point effects and multiple fuel options. The cost
model integrates the valve-point loadings and the fuel changes in
one frame. So, the nonsmooth cost function  is given as \cite{A42}:
\begin{equation*}\label{fuel}
 C(P)=\sum_{i=1}^n C_i(P_i)= \sum_{i=1}^n  a_{i}P_i^2 + b_{i}P_i + c_{i} + \left|g_{i}sin(h_{i}(P_{i}^{min} - P_i)) \right|\\
\end{equation*}
such that :  $ P_{i}^{min}\leq P_{i} \leq P_{i}^{max}$\\
where $g_i$, $h_i$, $a_{i}$, $b_{i}$ and $c_{i}$ are the cost
coefficients of generator $i$.\\
\subsection{Emission function}
The emission function can be formulated as the sum of all types of
emission considered, with convenient pricing
or weighting on each  emitted pollutant. In this paper, only one type
of emission NOx is taken into account without loss of generality
\cite{A99}. The volume of NOx emission is given as a function of
generator output. That is, the sum of a quadratic and exponential
function. The total amount of emission such as SO2 or NOx
depends on the amount of power generated by unit. The NOx emission
amount which is, the sum of a quadratic and exponential function can
be realistically written as :
\begin{equation*}
E(P)=\sum_{i=1}^n E_i(P_i)= \sum_{i=1}^n  \alpha_{i}P_i^2 + \beta_{i}P_i + \gamma_{i}  %\eta_i \exp(\delta_i P_i)
\end{equation*}
where, $\alpha_{i}$, $\beta_{i}$, $\gamma_{i}$, $\eta_i$ and $\delta_i$ are the coefficients of the ith
generator emission characteristics.\\
Formulation of the Non-Smooth EEDP:
 \[
 P_{\scalebox{0.5}{NEEDP}}: ~\left\{
   \begin{array}{ll}
   \displaystyle \min\{ C(P), E(P)\}\\
\text{subjet to}:~ \sum_{i=1}^n P_i = P_d\\
  P_i^{min}  \leq P_i \leq P_i^{max}
\end{array}
 \right.
 \]\
\subsection{Application: }
In order to solve EEDP with two generators we consider the
following problem :

\[
 P_{\scalebox{0.5}{NEEDP}}: ~\left\{
   \begin{array}{ll}
   \displaystyle \min\{ C(P), E(P)\}\\
\text{subjet to}:~ P_1 + P_2 = 650 \\
100 \leq P_1 \leq 600 \text{~ and~}  100 \leq P_2 \leq 400.
\end{array}
 \right.
 \]\

where \\
$C_1 (P_1) = 0.001562 P_1^2 + 7.92 P_1 + 561 + |300 \sin(0.0315(P^{min} - P_1 ))|$;\\
$C_2(P_2) = 0.00194 P_2^2 + 7.85 P_2 + 310 + |200 \sin(0.042(P^{min} - P_2 ))|$ ;\\
$C (P) = C_1 (P_1) + C_2 (P_2)$ ;\\
$E_1 ( P_1)= 0.0126 P_1^2 +  1.355 P_1 + 22.983$ ;\\
$E_2 (P_2 ) 0.00765  P_2^2 + 0.805  P_2 + 363.70$  ;\\
$E ( P) = E_1 ( P_1) + E_2 ( P_2)$.
\\

\begin{description}
    \item \emph{Step 1}: We apply the  smoothing method to the nonsmooth objective function
$\ C(P)$  (see Propositions \ref{p10} and \ref{p9}) to obtain a smooth objective function $\tilde{C}(P,{\mu})=\{ \tilde{C_1}(P_1,{\mu}),\tilde{C_2}(P_2,{\mu})\}$ as follows:\\

\begin{equation}
   \begin{split}
      \tilde{C_1}(P_1,{\mu})=&0.001562 P_1^2 + 7.92 P_1 + 561 +
\sin(\mu)\left[\ln\left(\frac{1}{2}\exp\left(\frac{200
\sin(0.042(P^{min} - P_2 ))}{\sin
(\mu)}\right)\right)\right]\\
        &+\frac{1}{2}\exp\left[(\frac{-200 \sin(0.042(P^{min} - P_2 ))}{\sin
(\mu)})\right]
   \end{split}
\end{equation}

\begin{equation}
   \begin{split}
     \tilde{C_2}(P_2,{\mu})=& 0.00194 P_2^2 + 7.85 P_2 + 310 +
\sin(\mu)\left[\ln\left(\frac{1}{2}\exp\left(\frac{300
\sin(0.0315(P^{min} - P_ ))}{\sin (\mu)}\right)\right)\right]\\
        & +\frac{1}{2}\exp\left(\frac{-300
\sin(0.0315(P^{min} - P_1 ))}{\sin (\mu)})\right].
   \end{split}
\end{equation}

 \item \emph{Step 2}: Each of $P^{\mu}_{\scalebox{0.5}{NEEDP}}$ subproblems has the form
 \[
 P^{\mu}_{\scalebox{0.5}{NEEDP}}: ~\left\{
   \begin{array}{ll}
   \displaystyle \min\{ \widetilde{C}(P), E(P)\}\\
\text{subjet to}:~ P_1 + P_2 = 650 \\
100 \leq P_1 \leq 600 \text{~ and~}  100 \leq P_2 \leq 400.
\end{array}
 \right.
 \]\
%\begin{equation*}
%tilde{C}_2(P_2,{\mu})= 0.00194 P_2^2 + 7.85 P_2 + 310 + \sin(\mu)\left[\ln\left(\frac{1}{2}\exp\left(sin(0.042(P^{min}  P_2 ))}{\sin (\mu)}\right)+\frac{1}{2}\exp\left(\frac{-sin(0.042(P^{min} - P_2 ))}{\sin (\mu)}\right)right)right]
%\end{equation*}
 \item \emph{Step 3}: The bi-objective  subproblem $P^{\mu}_{\scalebox{0.5}{NEEDP}}$  is transformed into a set of single-objective
subproblems using  the $\epsilon$-constraint method. For both
methods, the objective function of the subproblems are smoothed by
the smoothing method and the subproblems are solved by the interior
point barrier method \cite{A14}.

  \[
 P^{\mu,~\epsilon_l}_{\scalebox{0.5}{NEEDP}}: ~\left\{
   \begin{array}{ll}
   \displaystyle \min \tilde{C}(P,{\mu}),\\
   \text{subject to~} P_1 + P_2 = 650,\\
       E(P) <\epsilon_l;\\
       100 \leq P_1 \leq 600 ,\quad 100 \leq P_2 \leq 400.\\
 \end{array}
 \right.
 \]\
 \item \emph{Step 4}: To create constraint bound vector, consider the number of Pareto points
n = 70; let $\tau=\frac{\ E^{\max}-E^{\min}}{n}$
and  $\epsilon_{l+1} = \epsilon_{l}+\tau$, $l= \{1,\cdots,n-1\}$ with  $\epsilon_{1}=E^{\min}$, and
solve each smoothed single-objective subproblem $P^{\mu,~\epsilon_l}_{\scalebox{0.5}{NEEDP}}$ by the interior point barrier method.
\end{description}
%\begin{figure}
%  \centering
%  % Requires \usepackage{graphicx}
%  \includegraphics[width=10cm]{esp_smooth.png}\\
%  \caption{ Pareto Front using Smoothing-$\epsilon-$constraint - Interior point  method}\label{exam12}
%\end{figure}
%
%\begin{figure}
%  \centering
%  % Requires \usepackage{graphicx}
%  \includegraphics[width=10cm]{genetic.png}\\
%  \caption{ Pareto Front using Genetic Multiobjective Optimization}\label{exam12}
%\end{figure}

\begin{figure}[!h]
\begin{minipage}{.45\linewidth}
%\begin{center}
\hspace{0cm}
\includegraphics[scale=0.35]{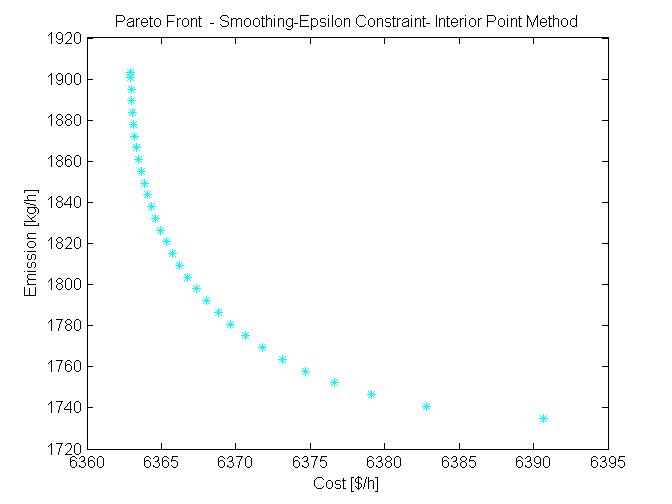}
\caption{\footnotesize{Pareto Front using Smoothing-$\epsilon-$constraint - Interior point  method for $\mu=0.000001$. }}
%\end{center}
\end{minipage}
\hfil
\begin{minipage}{.45\linewidth}
%\begin{center}
\hspace{0cm}
\includegraphics[scale=0.35]{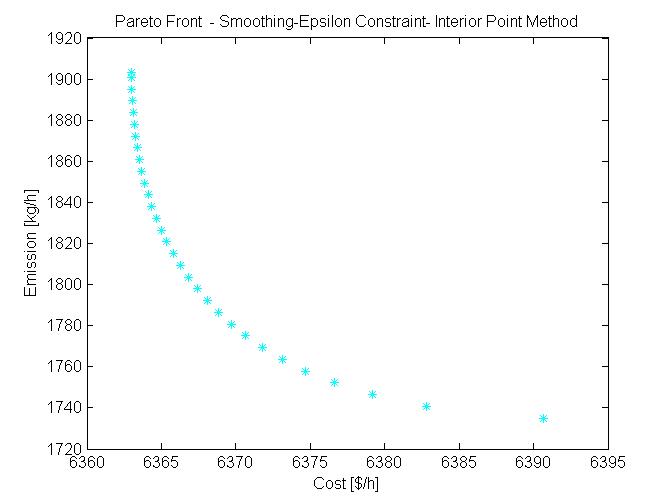}
\caption{\footnotesize{Pareto Front using Smoothing-$\epsilon-$constraint - Interior point  method for $\mu=0.0001$.}} \label{Afig1}
%\end{center}
\end{minipage}
\end{figure}

\begin{figure}[!h]

\begin{minipage}{.45\linewidth}
%\begin{center}
\hspace{0cm}
\includegraphics[scale=0.35]{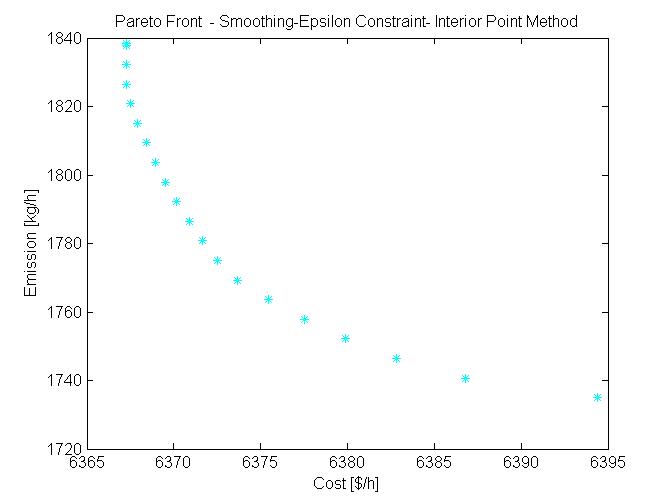}
\caption{\footnotesize{Pareto Front using Smoothing-$\epsilon-$constraint - Interior point  method for $\mu=0.01$}} \label{Afig1}
%\end{center}
\end{minipage}
\hfil
\begin{minipage}{.45\linewidth}
%\begin{center}
\hspace{0cm}
\includegraphics[scale=0.35]{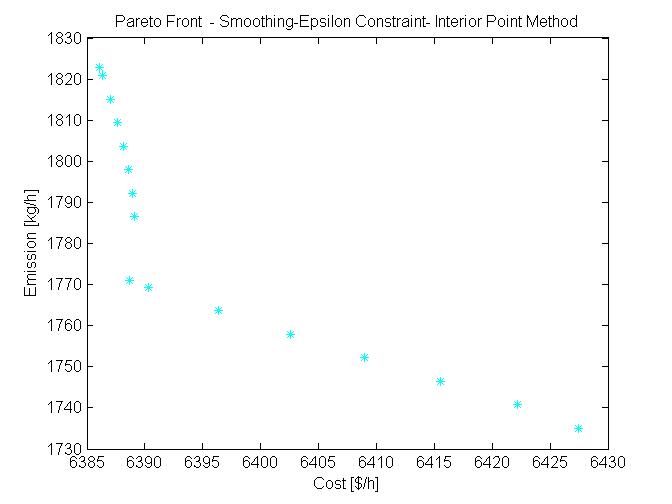}
\caption{\footnotesize{Pareto Front using Smoothing-$\epsilon-$constraint - Interior point  method for $\mu=0.1$.}} \label{Afig1}
%\end{center}
\end{minipage}
\end{figure}

\begin{figure}[!h]
\begin{minipage}{.45\linewidth}
%\begin{center}
\hspace{0cm}
\includegraphics[scale=0.35]{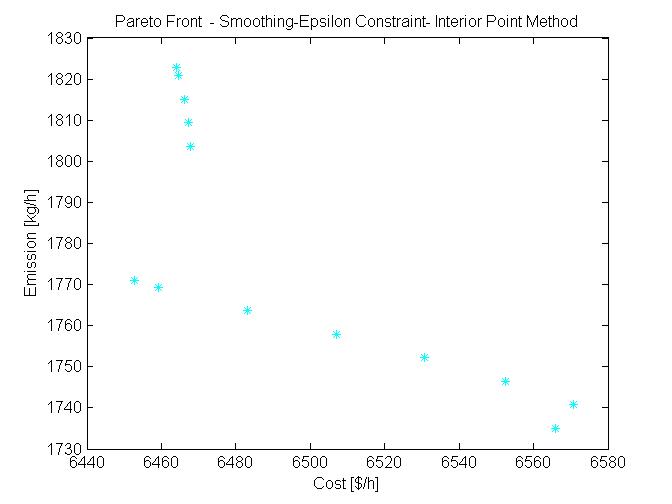}
\caption{\footnotesize{Pareto Front using Smoothing-$\epsilon-$constraint - Interior point  method for $\mu=0.5$.}} \label{Afig1}
%\end{center}
\end{minipage}
\hfil
\begin{minipage}{.45\linewidth}
%\begin{center}
\hspace{0cm}
\includegraphics[scale=0.35]{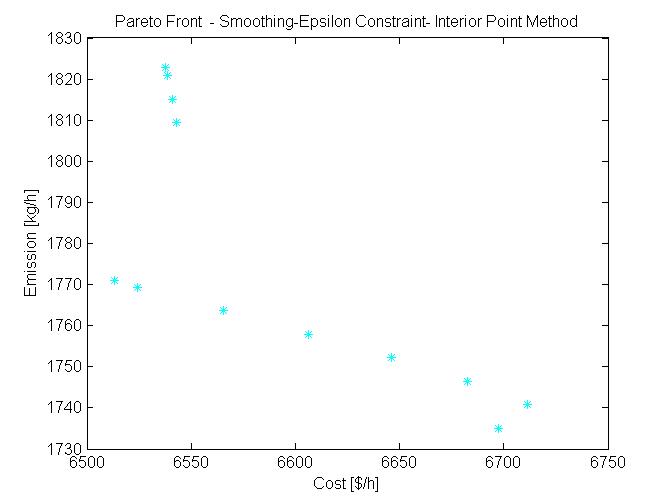}
\caption{\footnotesize{Pareto Front using Smoothing-$\epsilon-$constraint - Interior point  method for $\mu=1$.}} \label{Afig1}
%\end{center}
\end{minipage}
\end{figure}

In order to verify our approach performance, some simulations were performed and the results were compared with the PBC-HS-MLBIC method developed by Gonçalves et al. \cite{gon}. This work was chosen because we use the same parameters input.
in the following table,  $\widehat{C}(P)$ and $\widehat{E}(P)$  represent respectively, the total cost and  total pollutant emission obtained by PBC-HS-MLBIC method  \cite{gon} for two generators. \\
From the results presented in table \ref{tab99}, some remarks are rised. The first is the quality of our methodology, as our output results are significantly better compared to that of Gonçalves.

\begin{table}[htbp]
  \centering
  \caption{Minimum fuel cost, minimum emission and comparison to output results of PBC-HS-MLBIC method.}\label{tab99}
    \begin{tabular}{rrrrr|rr}
    \hline  $ n$& $P_1$(MW)& $P_2$(MW) & $C(P)$ (\$)& $E(P)$(kg/h) &$\widehat{C}(P)$ (\$) & $\widehat{E}(P)$(kg/h) \\
  \hline  1     & 259,0835 & 390,7781 & 6390,6884 & 1735,0267 & 6698,45 & 1737,14  \\
    2     & 274,8133 & 375,1867 & 6382,8279 & 1740,72 & 6724,26 & 1747,20\\
    3     & 282,1196 & 367,8804 & 6379,16 & 1746,44 &6681,07 & 1752,92\\

      .    & . & . & . & . & . & . \\
    .    & . & . & . & . & . & . \\
  .    & . & . & . & . & . & . \\
    35    & 350,0203 & 299,9797 & 6362,9941 & 1903,0076 & 6848,11 &  1936,07 \\
    36    & 350,0236 & 299,9764 & 6362,9941 & 1903,0198 & 6845,91 & 1941,79 \\
    37    & 350,0243 & 299,9757 & 6362,9941 & 1903,0223  & 6842,53 & 1947,51\\
       .    & . & . & . & . & . & . \\
    .    & . & . & . & . & . & . \\
  .    & . & . & . & . & . & . \\
    68    & 350,0243 & 299,9757 & 6362,9941 & 1903,0225  & 6407,48 & 2124,93 \\
    69    & 350,0243 & 299,9757 & 6362,9941 & 1903,0223  & 6389,74 & 2130,66\\
    70    & 350,0242 & 299,9758 & 6362,9941 & 1903,0222 &  6383,05 & 2135,34\\
    \hline
    \end{tabular}%
  \label{tab:addlabel}%
\end{table}%
\newpage
{\bf Conclusion:} \\
In this paper, a class of nonsmooth multiobjective optimization problems that include $\min$, $\max$,  absolute value functions or composition of the plus function $(t)^+$ with smooth functions is introduced, and some smoothing methods are presented. The algorithm is based on smoothing techniques to approximate the objective functions in all points where the function is nonsmooth. Numerical results show that the smoothing methods are promising for the nonsmooth MOP.

\newpage

 \end{sloppypar}

 \newpage
%%%%%%%%%%%%%%%%%%%%%%%%%%%%%%%%%%%%%%%%%%%%%%%%%%%%%%%%%%%%%%%%%%%%%%%%%%%%                              %
%%%%%%%%%%%%%%%%%%%%%%%%%%%%%%%%%%%%%%%%%%%%%%%%%%%%%%%%%%%%%%%%%%%%%%%%%%%%

%%%%%%%%%%%%%%%%%%%%%%%%%%%%%%%%%%%%%%%%%%%%%%%%%%%%%%%%%%%%%%%%%%%%%%%%%%%%
%               REFERENCES BIBLIOGRAPHIQUES                                %
%%%%%%%%%%%%%%%%%%%%%%%%%%%%%%%%%%%%%%%%%%%%%%%%%%%%%%%%%%%%%%%%%%%%%%%%%%%%
\bibliographystyle{plain}

%%%%%%%%%%%%%%%%%%%%%%%%%%%%%%%%%%%%%%%%%%%%%%%%%%%%%%%%%%%%%%%%%%%%%%%%%%%

\end{document}